\documentclass[11pt, twoside, a4paper]{article}
\usepackage[english]{babel}
\usepackage{amsmath,amssymb,amsthm,epsfig, color}
\usepackage[latin1]{inputenc}
\usepackage{amsfonts, eufrak, color}

\def\supp{\mathop{\rm supp}\nolimits}
\newtheorem{theorem}{Theorem}

\newtheorem{lemma}[theorem]{Lemma}
\newtheorem{proposition}[theorem]{Proposition}
\newtheorem{corollary}[theorem]{Corollary}
\newtheorem{definition}[theorem]{Definition}
\newtheorem{remark}[theorem]{Remark}
\newtheorem{example}[theorem]{Example}

\renewenvironment{proof}[1][.]{%
\bigskip\noindent{\bf Proof#1 }}{%
\hfill$\square$\bigskip}
\begin{document}
\pagestyle{myheadings}
\markboth{\textsc{\footnotesize{A. Baraviera, E. R. Oliveira  and F. B. Rodrigues}}}{\textsc{\footnotesize{On the dynamics of a rational semigroup}}}
\title{On the dynamics of a rational semigroup on a convolution measure algebra}
\author{Baraviera, A. T. ,  Oliveira, E. R., Rodrigues, F. B.}
\maketitle

\begin{abstract}
We are going to study the dynamical properties of the rational semigroup $Q_{t}(\mu)$ where
$Q_{t}(\mu)= (1-t) \mu * (1- t \mu)^{-1},$  for $t \in [0,1)$, that is defined for  $\mu \in \mathcal{P}(G)$, the set of Borel probabilities over $(G, \cdot)$ an abelian compact topological group where we define the \textbf{convolution}, $\nu * \mu \in \mathcal{P}(G)$, as usual for a group
$\int f d(\nu * \mu)= \int \int f(xy) d\nu(x) d\mu(y),$
then $(\mathcal{P}(G), *)$ became a \textbf{convolution measure algebra} (CM-algebra). We investigate several properties for this semigroup (as the Stable Manifold Theorem, Asymptotic behavior, invariant sets, differential properties, stationary points, etc) and how they are related with the Choquet-Deny equation. As an application we give a complete description of this semigroup for finite abelian groups.
\end{abstract}

\tableofcontents

\section{Introduction}
\index{Introduction}
As a motivation we consider the classical Choquet-Deny equation (see \cite{CD} for additional definitions and results), $\nu*\mu=\mu$, for a fixed $\nu$. Such equation is fundamental for harmonic analysis on groups, in probability theory for stochastic process behavior, etc. In order to obtain solutions of this equation, one can consider for a fixed $\nu \in \mathcal{P}(G)$ and $t \in (0,1)$, a modified version of the Choquet-Deny equation given by
$(1-t)\nu +t \nu*\mu=\mu.$
The parameter $t$ defines the solution $\mu_{t}$ that solves the  Choquet-Deny equation, $\nu*\mu=\mu$ when $t \to 1^-$. The modified equation can be solved by the following formal computation
$$(1-t)\nu +t \nu*\mu=\mu \Leftrightarrow (1-t)\nu*(1- t\nu)^{-1}=\mu.$$
That is, $\mu_{t}=(1-t)\nu*(1- t\nu)^{-1}=\frac{(1-t)\nu}{1- t\nu}$ is a rational function that solves the discounted Choquet-Deny equation (see Proposition~\ref{dyn CD eq} for a formal proof of this fact). In order to make precise these computations we introduce the idea of rational maps on CM-algebras in the Section~\ref{CM-algebras}.

For compact topological (not necessarily finite) groups we provide a tool that we call Choquet-Deny Kernel to characterize the limit sets of probabilities that are positive in open sets, Theorem~\ref{haar}.

The rest of the text is dedicated to understand the main properties of the rational semigroup  $\mu_{t}, \, t \in [0,1)$, as the existence of an infinitesimal generator and differential properties. Finally, we use this tools to understand the asymptotic behavior of $\mu_{t}$. The main achievement here are the proof of the Stable Manifold Theorem  (Theorem~\ref{hyp stab manif general}) that provides the convergence rate and the complete characterization of the local structure of  hyperbolic sets close to the stationary points. Additionally, in Section~\ref{basicsets}, we give a complete description of the basic hyperbolic sets when $G$ is a finite group making use of the techniques developed in our recently published work \cite{BOR} on the asymptotic behavior of  convolution powers for finite groups. As examples we give a complete description of the basic sets and the dynamics for order 4 groups: $\mathbb{Z}_4$ and the Klein's group.

\section{Maps on CM-algebras}\label{CM-algebras}
\index{Maps on CM-algebras}

Now we are going to recall the basic properties of CM-algebras and its maps.

\subsection{Convolution measure algebra}
Let $(G, \cdot)$ be an \textbf{abelian} compact topological group. We denote
$$BM(G)=\{\text{ Borel signed measures on }G\},$$
$$BM^{+}(G)=\{\text{ Positive Borel measures on }G\},$$
$$BM^{1}(G)=\{\text{ Positive Borel measures on }G\text{ such that } \mu(1)=1 \}=\mathcal{P}(G),$$

The usual operations make $BM(G)$ a linear normed space, and $\mathcal{P}(G) \subset BM^{+}(G) \subset BM(G)$ is a compact subset.  We introduce, as usual, a product $*$ in $BM(G)$ that is the convolution
$$\int f d(\nu * \mu)= \int \int f(xy) d\nu(x) d\mu(y), \forall f \in C^0(G).$$
The \textbf{convolution} $\nu * \mu \in BM(G)$ then $(BM(G),+, *)$ is a \textbf{convolution measure algebra}(usually named \emph{CM-algebra}, see \cite{Bob})  with unity, that is, there is $ 1 \in BM(G)$, given by $1= \delta_{e}$, where $eg=ge=g, \, \forall g \in G$, that is
$$1 * \mu= \mu * 1= \mu,  \, \forall \mu  \in BM(G).$$

\begin{remark}
Even if $G$ is a locally compact group and the measures are complex valuated, the CM-algebra, still a Banach algebra under the norm of total variation\footnote{We recall the definition of the total variation metric,
$$
\|\mu-\nu\|=\|\mu-\nu\|_{TV}=2\sup_{A\in \text{Borel}(G)}|\mu(A)-\nu(A)|= \sup_{|f| \leq 1}\left|\int_{G} f d\mu \right|.
$$} (see \cite{J.Wermer,Rudin}). The unitary ball in $ BM(G)$ is just weak* compact, but we can work because the total variation topology is stronger.
\end{remark}

We said that $\nu$ is an inverse of $\mu$ if  $\nu*\mu=1$. The restriction of $(BM(G),+, *)$ to $(\mathcal{P}(G),+, *)$ is a weak* compact  set closed by convolution with unity, since $1 \in \mathcal{P}(G)$ and $\mathcal{P}(G)$ is closed under convolution.

Let $\mu \in \mathcal{P}(G)$ be a probability, we remember that the support of $\mu$ is the set
$\supp \mu = \overline{\bigcup_{\mu(E)>0} E}$.  In particular if $G$ is finite, $\supp \mu=\{ g_i \;|\; \mu(g_i)>0\}$. The support of a probability is contained in a subgroup invariant under the convolution.

\begin{lemma}\label{invar of supp}
If $H$ is a closed subgroup of $G$ (always if $G$ is finite) and $\supp \mu   \subseteq H$, $  \supp \nu  \subseteq H$ then  $\supp \mu * \nu  \subseteq H$.
\end{lemma}
\begin{proof}
Indeed, from  \cite{Pym}, Lemma 1, we get for probability measures on compact semitopological semigroups (in particular for compact groups) that
$$\supp \mu * \nu = \overline{\supp \mu \cdot \supp \nu},$$
what concludes the proof because $H$ is closed.
\end{proof}

\subsection{Formal series on CM-algebras}
On a CM-algebra we can define formal series, that converges in the weak* topology
$$a_0 + a_1 \mu + a_2 \mu^2 + ...$$
where $\sum a_n < \infty $ and $\mu^n =  \mu* \mu* ...* \mu$ and $\mu^0 =\delta_{e}=1$. In the next we will always denote $\delta_{e}=1$ where $e$ is the identity  of $G$ in order to simplify the computations. More generally, if $\displaystyle F(x)= \sum_{n=0}^{\infty} a_n x^n$ is a power series with positive convergence radius $R$ then
$$F(\mu)= \sum_{n=0}^{\infty} a_n \mu^{n},$$
defines a function in $F: BM(G) \to BM(G)$.  As follows, if the coefficients $a_n$ are  nonnegative then $F: BM^{+}(G) \to BM^{+}(G)$ and, if $F(1)=1$ then $F:\mathcal{P}(G) \to \mathcal{P}(G)$.

\begin{example}
Given $t \in (0,1)$ the function
$F(x)=(1-t)(1+  t  x +   t^2  x^2 +   t^3 x^3 +...) = \frac{(1-t)}{1- t x}$ is a positive geometric series.
Since $F(1)=1$,  $F$ defines a map $F_{t}: \mathcal{P}(G) \to \mathcal{P}(G)$,
$F_{t}(\mu)=(1-t)(1+  t  \mu +   t^2  \mu^2 +   t^3 \mu^3 +...)$
By analogy denote $F_{t}$ the geometric map in $\mathcal{P}(G)$.
\end{example}

\begin{example}
Other formal series in $\mathcal{P}(G)$, can be obtained by normalizing the power series of the exponential function, for instance
$\displaystyle F(x)= \frac{1}{e } \sum_{n=0}^{\infty} \frac{x^n}{ n!} = e^{(x-1)}$ thus $F(1)=1$ and
$\hat{e}(\mu)= \frac{1}{e}(1 +  \frac{\mu}{1!} + \frac{\mu^{2}}{2!} + \frac{\mu^{3}}{3!} + ...),$
defines $F(\mu)= \hat{e}(\mu)$ the exponential map in $\mathcal{P}(G)$.
\end{example}

\subsection{Invertible elements and rational maps}
The set $\mathcal{P}(G)$ has an affine structure so it is natural to consider segments from an invertible element. Given $t \in [0,1)$ and $\mu \in \mathcal{P}(G)$, then  $\nu= 1 - t \mu \in BM(G)$ has an  inverse,  given by the Liouville series
$$ 1+  t  \mu +  t^2  \mu^2 +  t^3 \mu^3 +...$$
denoted $\nu^{-1}$. Indeed,
$\nu * \nu^{-1} =(1 - t \mu) * (1+  t  \mu +  t^2  \mu^2 +  t^3 \mu^3 +...) = 1.$
From this we can define rational maps and to solve certain  linear equations in $\mathcal{P}(G)$.
We start by defining polynomial functions. A polynomial $S$ in $\mathcal{P}(G)$ is an expression $S(\mu)=a_0 1 + a_1 \mu + a_2 \mu^2 + ... + a_m \mu^m$
where $a_0  + a_1   +   ... + a_m =1$ an d $a_i \geq 0, \; i=0, ..., m$.
Given $t \in [0,1)$ and $S_1(\mu), S_2(\mu)$ polynomials we define the \textbf{rational function} $F_{t}: \mathcal{P}(G) \to \mathcal{P}(G)$ by, $$F_{t}(\mu)= (1-t) S_1(\mu) * (1- t S_2(\mu))^{-1} ,$$
that is well defined.

\begin{remark}
For algebraic purposes, is convenient, to denote
$$(1-t) S_1(\mu) * (1- t S_2(\mu))^{-1}= \frac{(1-t) S_1(\mu)}{1- t S_2(\mu)},$$
and from this point the quotient $\frac{a}{b}$ means $a * b^{-1}$ in the CM-algebra.
\end{remark}

\begin{example} Consider the fraction $F_{t}(\mu)=\frac{(1-t)}{1- \frac{t}{2}  - \frac{t}{2} \mu^2}$ then  $S_1(\mu)= 1$ and $S_2(\mu)=1/2 ( 1 + \mu^2)$.  Thus, we have the rational map $$F_{t}(\mu)= (1-t) * (1- (t/2) ( 1 + \mu^2))^{-1}.$$
\end{example}

\section{The rational semigroup}
\subsection{Definition of  the rational map $T$}

\begin{proposition}\label{dyn CD eq}
The modified Choquet-Deny equation $(1-t)\nu +t \nu*\mu=\mu$ has a unique solution $\mu_{t}$ that is a rational function of $\nu$.
\end{proposition}
\begin{proof}
We introduce the map $F(\mu)= (1-t) \nu + t \nu*\mu$
and consider a dynamical argument. We claim that $F$ has a unique attracting point denoted by $\mu_{t}=Q_{t}(\nu)$, so this is the unique fixed point.

By iterating $F$ from any initial point $\mu$\\
$F(\mu)= (1-t) \nu + t \nu*\mu,$\\
$F^{2}(\mu)= (1-t) \nu + \beta  \nu*((1-t) \nu + t \nu*\mu)= (1-t) \nu + (1-t) t \nu^2 + t^2 (\nu^2 * \mu)$.\\
$F^{3}(\mu)= \dots$ we get
$$F^{n}(\mu)=  (1-t) \nu + (1-t) t  \nu^2 + ... + t^n (\nu^n * \mu) \to Q_{t}(\nu),$$
when $n \to \infty$, because $t^n \to 0$.
Since the unique limit point is invariant, $F(Q_{t}(\nu))=Q_{t}(\nu)$, what means that $$\mu_{t}= Q_{t}(\nu)=(1-t) \mu + (1-t) t  \mu^2 + (1-t) t^2  \mu^3 +...$$ solves the previous  equation \textbf{$(1-t) \nu + t \nu*\mu_{t} = \mu_{t}$}.
\end{proof}

In that case any accumulation point of the one parametric family $\mu_{t}, \; t \in [0,1)$, will be a solution of the Choquet-Deny Equation if $t \to 1^-$, so is natural to introduce the  rational map.

\begin{definition}The  \textbf{rational map} $T:\mathcal{P}(G) \to \mathcal{P}(G)$ is a Moebius map for  a fixed $t \in[0,1)$, $T(\mu)= Q_{t}(\mu)$ given by the fraction
$$T(\mu)=\frac{(1-t) \mu}{1- t \mu} $$
because $(1-t) \mu * (1- t \mu)^{-1} =(1-t) \mu + (1-t) t  \mu^2 + ...=Q_{t}(\mu).$
\end{definition}

\subsection{Map properties}
We are going to consider the main properties of the rational map, as a function.

\begin{lemma}\label{AA} For any $\mu, \nu \in \mathcal{P}(G)$,
  $$Q_{t}(\mu) - Q_{t}(\nu)= (1-t)*( 1- t \mu)^{-1}*( 1- t \nu)^{-1} *( \mu  -   \nu )$$
\end{lemma}
\begin{proof} Let $\mu, \nu \in \mathcal{P}(G)$ be fixed probabilities on a commutative group, we want to estimate $Q_{t}(\mu) - Q_{t}(\nu) \in BM(G)$.

\begin{align*}
&Q_{t}(\mu) - Q_{t}(\nu)= (1-t) \mu *( 1- t \mu)^{-1} -  (1-t) \nu *( 1- t \nu)^{-1}
\\&\Rightarrow[Q_{t}(\mu) - Q_{t}(\nu)]*( 1- t \mu)*( 1- t \nu)= (1-t)[ \mu *( 1- t \nu) -   \nu *( 1- t \mu)]
\\&\Rightarrow[Q_{t}(\mu) - Q_{t}(\nu)]*( 1- t \mu)*( 1- t \nu)= (1-t)[ \mu  -   \nu ]
\\&\Rightarrow Q_{t}(\mu) - Q_{t}(\nu)= (1-t)*( 1- t \mu)^{-1}*( 1- t \nu)^{-1} ( \mu  -   \nu )
\end{align*}
\end{proof}

\begin{proposition} \label{rationalmapproperties}
The rational map $T$ has the following properties\\
a)T is injective,\\
b)T is not surjective, except for $t=0$,\\
c)T is Lipschitz.
\end{proposition}
\begin{proof}
a) We need to solve the equation
$\mu_{1} * (1- t \mu_{1})^{-1}=\mu_{2} * (1- t \mu_{2})^{-1}$
for each fixed $t$ with respect to $\mu_{i}  \in\mathcal{P}(G)$.
\begin{align*}
\mu_{1} * (1- t \mu_{1})^{-1}&=\mu_{2} * (1- t \mu_{2})^{-1}\\
\mu_{1} * (1- t \mu_{2})&=\mu_{2} *(1- t \mu_{1})\\
\mu_{1} - t \mu_{1} *\mu_{2}&=\mu_{2} - t \mu_{2} *\mu_{1}\\
\mu_{1} &= \mu_{2}
\end{align*}

b) If we suppose that $T$ is surjective then we can solve the equation
$(1-t) \mu * (1- t \mu)^{-1}= \nu$ or equivalently $[(1-t) \delta_{e} + t \nu ] * \mu = \nu$
for each fixed $t$  with respect to $\mu$ for all $\nu \in\mathcal{P}(G)$. However, it is not possible even for finite groups.
For example, if $G= Klein=\{e,a, b, ab\}$ we can take $\nu=\delta_{b}$ then the equation above for $0<t<1$ is
$$[(1-t) \delta_{e} + t \nu ] * \mu = \nu $$
$$[(1-t) \delta_{e} + t \delta_{b} ] * (x \delta_{e} + y \delta_{a} +z  \delta_{b}+w  \delta_{ab}) = \delta_{b}$$
$$((1-t)x + tz) \delta_{e} + ((1-t)y + tw) \delta_{a} +((1-t)z+tx)  \delta_{b}+((1-t)w + ty) \delta_{ab}= \delta_{b}.$$
From the first coordinate we get $(1-t)x + tz=0$. Or $x=z=0$, what is impossible because $(1-t)z+tx=1$, or $z= \frac{-t}{(1-t)} x <0$ what means that $\mu=x \delta_{e} + y \delta_{a} +z  \delta_{b}+w  \delta_{ab}$ is not a probability measure.

c) From Lemma~\ref{AA} we have
$$Q_{t}(\mu) - Q_{t}(\nu)  = (1-t) (\mu  - \nu) (1- t \mu)^{-1}(1- t \nu)^{-1}.$$
So,
$ \| Q_{t}(\mu) - Q_{t}(\nu) \|  = (1-t)\|(1- t \mu)^{-1}(1- t \nu)^{-1} \| \| \mu  - \nu \|,$
since $\|(1- t \mu)^{-1}\| \leq \frac{1}{(1-t)}$, we get
$\| Q_{t}(\mu) - Q_{t}(\nu) \|  \leq \frac{1}{(1-t)} \| \mu  - \nu \|.$
Thus $T= Q_{t}$ is Lipschitz continuous.
\end{proof}

\section{Dynamics of the rational semigroup}
In this section we are going to define the rational semigroup from the rational map and  study its dynamical properties.

\subsection{Semigroup properties}
From the iteration of $T=Q_{t}$ we have $Q_{t}^{0}(\mu)= \mu, \; Q_{t}^{1}(\mu)= Q_{t}(\mu),\; Q_{t}^{2}(\mu)= Q_{t}(Q_{t}(\mu)), etc$. In order to understand these iterates we introduce an operation $\Delta$ on the interval [0,1) and show $([0,1),\Delta)$
is actually a semigroup.

\begin{definition} The $\Delta$-semigroup, $([0,1), \triangle)$ is defined by the operation $\triangle : [0,1)\times [0,1) \to [0,1)$, given by
$$t_1 \triangle t_2 = 1- (1- t_1)(1- t_2).$$
\end{definition}

We can also consider the extended $\Delta$-semigroup, $([0,1], \triangle)$.  Here the value $t=1$ plays the role of $+\infty$ for positive additive semigroups $([0, \infty), +)$ that usually appears in the study of differential equations flows  on $\mathbb{R}$.

\begin{proposition}\label{Delta properties} The following properties are true for the $\Delta$-semigroup.\\
\begin{enumerate}
  \item $t \triangle 0 = t $ (on the extended $\Delta$-semigroup)
  \item $t \triangle 1 = 1 $ (on the extended $\Delta$-semigroup)
  \item $t_1 \triangle t_2 = t_2 \triangle t_1 $
  \item $(t_1 \triangle t_2) \triangle t_3 = t_1 \triangle (t_2 \triangle t_3) $
\end{enumerate}
\end{proposition}
\begin{definition} The  \textbf{rational semigroup} $Q_{t}: [0,1) \times \mathcal{P}(G) \to \mathcal{P}(G)$ is the one parameter family defined by the $\Delta$-semigroup $([0,1), \triangle)$ on $\mathcal{P}(G)$ given by
$$(t, \mu)   \to Q_{t}(\mu).$$
\end{definition}
\begin{remark}
From now we will understand the rational map, $T= Q_{t},$ as the time $t$ iteration of the one parameter family  $Q_{t}$.
\end{remark}
\begin{lemma}\label{flow property} The one parameter family $Q_{t}$  has the semigroup property\\
$$Q_0 \mu=\mu \text{ and } Q_{t_1} \circ Q_{t_2}(\mu)=Q_{t_1 \triangle t_2}(\mu).$$
\end{lemma}
\begin{proof}
Using the properties of the $\Delta$-semigroup we get,
\begin{align*}
&Q_{t_1}(Q_{t_2}(\mu))=(1-t_1)  Q_{t_2}(\mu) *( 1- t_1 Q_{t_2}(\mu))^{-1}
\\&\Rightarrow
Q_{t_1}(Q_{t_2}(\mu))*( 1- t_1 Q_{t_2}(\mu))=(1-t_1) Q_{t_2}(\mu)
\\&\Rightarrow
Q_{t_1}(Q_{t_2}(\mu))*( 1- t_1 ((1-t_2) \mu *( 1- t_2 \mu)^{-1}))=(1-t_1)  (1-t_2) \mu *( 1- t_2 \mu)^{-1}
\\&\Rightarrow
Q_{t_1}(Q_{t_2}(\mu))*( 1- (t_1 + t_2 - t_1  t_2) \mu ) =(1-t_1)  (1-t_2) \mu.
\end{align*}
Thus
\begin{align*}
Q_{t_1}(Q_{t_2}(\mu)) &=(1-t_1)  (1-t_2) \mu *( 1- (1-(1-t_1)  (1-t_2)) \mu )^{-1}  \\
Q_{t_1}(Q_{t_2}(\mu)) &=(1- t_1 \triangle t_2) \mu *( 1- t_1 \triangle t_2 \mu )^{-1}\\
Q_{t_1}(Q_{t_2}(\mu)) &=Q_{t_1 \triangle t_2}(\mu).
\end{align*}
\end{proof}

We remind the reader that the flow $Q_{t}$ is complete if its flow curves exist for all time $t \in [0,1)$ generating a family of homeomorphisms $Q_{t}: \mathcal{P}(G) \to \mathcal{P}(G)$.
\begin{proposition}\label{flow complete} The one parameter family $Q_{t}$  is a complete flow.
\end{proposition}
\begin{proof} For all time $t \in [0,1)$  the integral curves, $t \to Q_{t}(\mu)$ are well defined  because the series is absolutely convergent. From Proposition~\ref{rationalmapproperties} we have that, for each fixed $t \in [0,1)$, $Q_{t}=T$ is injective, so it is a homeomorphism on its image.
\end{proof}

\subsection{Asymptotic behavior of the rational semigroup}
The asymptotic behavior of the rational map is contained in the limit of the rational semigroup when $t \to 1^-$.

\begin{proposition}\label{Iterate Q} The iterates of the rational map are given by
$Q_{t}^{n}(\mu)=Q_{t^{n_{\triangle}}}(\mu),$
in other words,
$Q_{t}^{n}(\mu)=(1-t)^{n} \mu*( 1- \left[1- (1-t)^{n}\right] \mu)^{-1},$
where $\stackrel{n \; \text{times}}{\overbrace{t \triangle \cdots  \triangle t}} = t^{n_{\triangle}}$.
\end{proposition}
\begin{proof}
The proof is by induction on $n$.  For $n=0$ is clear that
$$
Q_{t}^{0}(\mu)= \mu= (1-t)^{0} \mu *( 1- \left[1- (1-t)^{0}\right] \mu)^{-1}= Q_{0}(\mu).
$$
For $n+1$,   supposing $Q_{t}^{n}(\mu)= Q_{t^{n_{\triangle}}}(\mu)$ we just apply Lemma~\ref{flow property} and the flow property,
\begin{align*}
Q_{t}^{n+1}(\mu)&=Q_{t^{n_{\triangle}}}(Q_{t}(\mu))
= Q_{t^{n_{\triangle}} \triangle t}(\mu) =Q_{t^{n+1_{\triangle}}}(\mu)
\\&=(1-t)^{n+1} \mu*( 1- \left[1- (1-t)^{n+1}\right] \mu)^{-1}
\end{align*}
\end{proof}

\begin{example} If we compose different times $t_n$ then it can accumulate on other times $t$ less than 1. For example, from Euler's infinite product theorem, we know that
$$\sin(z)=z \prod_{n=1}^{\infty} (1-\frac{z^2}{\pi^2 n^2}), \forall z \in \mathbb{C},$$
and so, taking $z=1$, we get $\displaystyle\lim_{n \to \infty} \prod_{j=1}^{n} (1-\frac{1}{\pi^2 j^2})=\sin(1)$, what means that
\begin{align*}
Q_{\frac{1}{\pi^2 n^2}}\circ Q_{\frac{1}{\pi^2 (n-1)^2}}\circ\cdots\circ Q_{\frac{1}{\pi^2 (1)^2}} &= Q_{\frac{1}{\pi^2 n^2} \triangle \frac{1}{\pi^2 (1)^2}\triangle  \cdots \triangle \frac{1}{\pi^2 (1)^2}}
\\&=
 Q_{1- \prod_{j=1}^{n} (1-\frac{1}{\pi^2 j^2})} \stackrel{n \to \infty}{\to}  Q_{1- \sin(1)}.
\end{align*}
\end{example}

From the previous discussion we know that $Q_{t}^{n}(\mu)=Q_{t^{n_{\triangle}}}(\mu),$ so the long time behavior of every $T(\mu)=Q_{t}(\mu)$, for a fixed $t$, is given by  the long time behavior of the flow $Q_{t}$. We recall that $\omega$-limit the of $\mu$ by $Q_{t}$ is the set of accumulation points of its orbit, and is denoted by
$$
L_{\omega}^{Q}(\mu)= \{\nu   \in\mathcal{P}(G) \;|\; \nu=\lim_{t_{k} \to 1} Q_{t_{k}}(\mu),\mbox{ for some sequence }, t_k\to1\},
$$
and
$$
L_{\omega}^{T}(\mu)= \{\nu   \in\mathcal{P}(G) \;|\; \nu=\lim_{k \to \infty} T^{n_{k}}(\mu)=\lim_{k \to \infty} Q_{t}^{n_{k}}(\mu)\}.
$$

\begin{proposition}\label{omega limite} Let $\nu   \in\mathcal{P}(G)$ be a probability measure then
$$L_{\omega}^{T}(\mu) \subseteq L_{\omega}^{Q}(\mu) \subseteq \{\nu   \in\mathcal{P}(G) \;|\; \nu*\mu=\nu\}.$$
\end{proposition}
\begin{proof}  The inclusion $L_{\omega}^{T}(\mu) \subseteq L_{\omega}^{Q}(\mu)$ is obvious because $Q_{t}^{n}(\mu)=Q_{t^{n_{\triangle}}}(\mu),$ and $t^{(n_{k})_{\triangle}} \to 1$ when $k \to \infty$.
Let $\nu \in L_{\omega}^{Q}(\mu)$, so  $\nu=\lim_{k \to \infty} \nu_{k}$ where  $\nu_{k} = Q_{t_{k}}(\mu)$.
Then
\begin{align*}
\nu_{k} = Q_{t_{k}}(\mu) = (1-t_{k}) \mu*( 1- t_{k} \mu)^{-1}
&\Rightarrow\nu_{k} *( 1- t_{k} \mu) = (1- t_{k} ) \mu
\\&\Rightarrow
\nu * ( 1-  \mu) = 0.
\end{align*}
Taking the limit above we get $\nu*\mu=\nu,$  that is, if $\nu \in L_{\omega}^{Q}(\mu)$ then $\nu$ \textbf{is solution of the Choquet-Deny equation $\nu*\mu=\nu$ for a fixed $\mu$} (this is the classical sense).
\end{proof}

From the Proposition~\ref{omega limite} we know that the $\omega$-limit is separated by the solution of the Choquet-Deny equation, so we can define the Choquet-Deny kernel
$$\ker (\nu_{0}) = \{\mu   \in\mathcal{P}(G) \;|\; \nu_{0}*(1-\mu)= 0\},$$
and the co-kernel
$$\ker^{c} (\mu_{0}) = \{\nu   \in\mathcal{P}(G) \;|\; \nu*(1-\mu_{0})= 0\}.$$
From the previous results we know that
$L_{\omega}^{T}(\mu) \subseteq L_{\omega}^{Q}(\mu) \subseteq \ker^{c} (\mu)\subseteq \ker^{c} (\mu^n), \; \forall n$ and $\ker (\nu_{0})$ is the set of probabilities $\mu$ that can share the same $L_{\omega}^{Q}(\mu)$.

\begin{remark}\label{rmk:S^1 gtoup}
Let $G=\mathbb{S}^1$ and $Leb$ the Lebesgue measure. Then, for any $\mu\in \mathcal{P}(\mathbb{S}^1)$,
$Leb\ast\mu=\mu\ast Leb  =Leb$. In particular $Leb \in Ker^{c}(\mu), \; \forall \mu$. Indeed, $$
 \mu\ast Leb(E)=\int_{\mathbb{S}^1}Leb(g^{-1}E)d\mu(g)=\int_{\mathbb{S}^1}Leb(E)d\mu(g)=Leb(E).
$$
\end{remark}

In \cite{Zeng}, Theorem 2, they prove that $\mu$ satisfy $\mu * \nu = \mu$ if, and only if, $\mu * \delta_{g} = \mu$, for all $g \in S(\nu)$, where $S(\nu)$ is the closed group generated by the support of $\nu$.  Ii is also easy to prove that, if $G$ is a topological metric space (in particular a finite group) a probability $\eta$ is a Haar measure if and only if  $\eta^2=\eta$ and $\eta$ is positive in open sets. Moreover, a probability $\lambda$ is a Haar measure if and only if  $\eta*\lambda=\lambda$ and $\eta$ is positive in open sets.

\begin{lemma}\label{choquet-deny is haar}
$\mu\ast\lambda=\lambda$ with $\mu$ such that $\mu(U)>0$ for any open set of G,
if and only if $\lambda(g^{-1}A)=\lambda(A)$ for any $g\in G$, i.e.,
$\lambda $ is a left invariant Haar measure. In particular, $\mu$ is positive in open sets then $\ker^{c}(\mu)=\{\lambda\}$.
\end{lemma}
\begin{proof}
For any bounded continuous function $u$ on $G$
$$
v(g)=\int_G u(gh)d\lambda(h)
$$
Then for any $a\in G$
 \begin{align*}
 \int_G v(ag)d\mu(g)&=\int_G\int_G u(agh)d\mu(g)d\lambda(h)
 \\& =\int_G u(ak)d(\mu\ast\lambda)(k)
 \\& =\int_G u(ak)d\lambda(k) = v(a)
\end{align*}
We can take $a$ as being the element where the maximum of $v$ is attained. Then $v(ag)=v(a)$ for all $g$ in the support of $\mu$.
 In particular $v(ag)$ and therefore $v$ is a constant. Hence, $\delta_{a}\ast\lambda = \lambda$ or $\lambda$ is left invariant.
\end{proof}

Here we show an important result which tell us how
the Haar measure is in a general group.
\begin{theorem}\label{equiv haar idemp}
The following are equivalent\\
1. $\lambda$ is a left invariant Haar measure on G.\\
2. $\lambda$ is a right invariant Haar measure on G.\\
3. $\lambda$ is an idempotent i.e $\lambda\ast\lambda=\lambda$ and $\lambda(U) > 0$ for every open set $U$.
\end{theorem}
\begin{proof}
If $\lambda$ is right invariant then $\lambda\ast\delta_g=\lambda$ for all $g\in G$ and by integrating $\lambda\ast\lambda=\lambda$ and that implies that $\lambda$ is left invariant as well. We note that any left or right invariant measure cannot give zero mass to any open set because, by compactness, $G$ can be covered by a finite number of translates of $U$.
\end{proof}

\begin{theorem}\label{haar} If $G$ is a topological metric  group (for example, any finite group) and $\mu^n$ is positive in open sets of the subgroup generated by its support (for example if $\mu$ is acyclic, see Definition~\ref{Acyclic})  then
$$L_{\omega}^{Q}(\mu)=\ker^{c}(\mu) \cap S(\mu) =\{\nu\},$$
where $\nu^2 =\nu$ is the Haar probability of the subgroup $S(\mu)$. In other words, the only $\omega$-limit points for such probabilities are the fixed points of $Q_{t}$.
\end{theorem}
\begin{proof} We know that
$L_{\omega}^{Q}(\mu) \subseteq \ker^{c} (\mu)\subseteq \ker^{c} (\mu^n), \; \forall n \in \mathbb{N}.$
From the hypothesis, we have $\mu^n$ positive in open sets for the induced topology of $G'=S(\mu)$ then from Lemma~\ref{choquet-deny is haar} $\ker^{c}(\mu^n)= \{\nu\}$ the Haar measure of $G'$  and by Theorem~\ref{equiv haar idemp} we conclude that $\nu^2=\nu$.
\end{proof}

\subsection{Invariant sets}
The first step is to determine the invariant sets  of $ T= Q_{t}$.

\begin{theorem}\label{invariantsetsubgroup}
If $H$ is a closed subgroup of $G$ and
$\mathbb{H}=\{ \mu \; |\; \supp \mu \subseteq H\}$
then $ Q_{t}(\mathbb{H}) \subseteq \mathbb{H}$.
\end{theorem}
\begin{proof}Indeed, from Lemma~\ref{invar of supp} we have,
$$Q_{t}(\mu)=  (1-t) \mu + (1-t) t  \mu^2 + ... \in  \mathbb{H},$$
because $\mu, \mu^2, \mu^3, ... \in \mathbb{H}$ if $\mu \in \mathbb{H}$.
\end{proof}

\begin{lemma}\label{kernel and co-kernel invariance} For each $\nu_0, \mu_0 \in \mathcal{P}(G)$, the sets $\ker (\nu_{0})$ and $\ker^{c} (\mu_{0})$ are invariant sets of $Q_{t}$.
\end{lemma}
\begin{proof}
It is easy to see that $\ker^{c} (\mu_{0})$ is an invariant subset by the rational semigroup $Q_{t}$. Indeed, take $\nu \in \ker^{c} (\mu_{0})$ then $Q_{t}(\nu)= (1-t)*\nu * (1- t \nu)^{-1}$ and it implies,
\begin{align*}
Q_{t}(\nu)* (1- t \nu)&= (1-t)*\nu
\Rightarrow Q_{t}(\nu)* (1- t \nu)*\mu_{0}= (1-t)*\nu*\mu_{0}
\\&\Rightarrow Q_{t}(\nu)* (\mu_{0}- t \nu*\mu_{0})= (1-t)*\nu*\mu_{0}
\\&\Rightarrow Q_{t}(\nu)* (\mu_{0}- t \nu)= (1-t)*\nu
\Rightarrow Q_{t}(\nu)*\mu_{0} =  Q_{t}(\nu),
\end{align*}
that is  $Q_{t}(\nu) \in \ker^{c}(\mu_{0})$. Analogously one can show that $Q_{t}(\ker  (\nu_{0})) \subseteq \ker  (\nu_{0})$.
\end{proof}

\begin{remark} Also, the set of fixed points  $Fix(Q_{t})=\{ \mu |  Q_{t}(\mu)=\mu\}$  is trivially invariant by the rational semigroup  $Q_{t}$.
The fixed points of the rational map are given by
\begin{align*}
T(\mu)&=Q_{t}(\mu)= \mu
\Rightarrow\mu= (1-t) \mu * (1- t \mu)^{-1}
\\&\Rightarrow\mu * (1- t \mu)= (1-t) \mu
\Rightarrow\mu^2=  \mu
\end{align*}
That is, $Fix(Q_{t})=\{\mu\,|\, \mu^2=  \mu\}$ is the set of idempotent measures. In \cite{Rudin}, this set, that is, measures such that $\mu^2 =\mu$ is characterized by combination of the Haar measures on compact subgroups, if $G$ is a topological locally compact group that is not discrete. For discrete groups the only solution is the counting measure but it is not a probability. Let us now consider what happens in a non-compact discrete group. We analyze only
the case of $\mathbb{Z}$ in order to illustrate what can happen.

For a given probability $\mu \in \mathcal{P}(\mathbb{Z})$
we can associate a function $f_{\mu} \colon [0, 2\pi] \to \mathbb{C}$
as follows
$$
   f_{\mu}(x) = \sum_{k \in \mathbb{Z}}  \mu(k) e^{ikx}.
$$

This function is continuous (use Weierstrass criteria); hence the Fourier
series converges at each point and in the compact set $[0, 2\pi]$ and the
convergence is uniform.

It is easy to see that $f_{\mu}(0) = \sum_{k \in \mathbb{Z}} \mu(k) = 1$
and, if the derivative of $f$ is defined, that $f'(0)$  corresponds to the mean value of the probability
$\mu$.
It is also true, and not hard to verify, that
$$
   f_{\mu}(x) f_{\nu}(x) = \sum_{k \in \mathbb{Z}}  (\mu*\nu)(k) e^{ikx}  = f_{\mu*\nu}(x)
$$

 Now, if we want to find a probability that satisfies the equation
 $\mu*\mu= \mu$ we can just convert the problem to the functional
 setting and we get the equation
$$
   f_{\mu(x)} = (f_{\mu}(x))^2
$$
and so, for any $x$ we have $f_{\mu}(x) = 0$ or $1$. In particular,
the constant function $f=1$ is a solution corresponding
to the probability measure $\delta_0$. We remark that this is the
only continuous solution of this equation that corresponds
to a probability (the other one being the constant zero function,
that corresponds to the null measure), showing that the
only probability $\mu$ in $\mathbb{Z}$ such that $\mu*\mu = \mu$ is the
Dirac measure $\delta_0$.
\end{remark}

\subsection{Differential properties: infinitesimal generator and hyperbolicity}

In this section we will differentiate the flow, with respect to $t$ and $\mu$ both under suitable hypothesis. The derivative in a commutative Banach algebra is defined in \cite{Blum}.
Let $z_0$ be a point in a commutative Banach algebra $B$ and let $N(z_0)$ be a neighborhood of $z_0$. Let $f(z)$ be a function defined on $N(z_0)$ with range in $B$. $f(z)$ is ``differentiable" at $z_0$ if there is an element, $f'(z_0)$, in $B$ with the property that for any real $\varepsilon >0$ there is a $\delta >0$ such that for all $z$ in $B$ with $\|z-z_0\| <\delta$
$\|f(z) - f(zo) - (z - z_0)f'(z_0) \|< \varepsilon \|z - z_0\|.$ The element $f'(z_0)$ is called the ``derivative" of $f(z)$ at $z_0$.

\begin{proposition}\label{AB} The derivative of $Q_{t}$ with respect to $\mu$ is
  $$Q_{t }'(\nu) =d_{\nu}Q_{t } =  (1-t)( 1- t \nu)^{-2} \in BM(G).$$
\end{proposition}
\begin{proof} Let $\nu\in \mathcal{P}(G)$ be a fixed probability and take $\eta=\mu-\nu$. From Lemma~\ref{AA}, we have
\begin{align*}
&\| Q_{t}(\mu) - Q_{t}(\nu) -  (1-t)( 1- t \nu)^{-2} * \eta\|
\\&=
\| (1-t)*( 1- t \mu)^{-1}*( 1- t \nu)^{-1} *\eta -  (1-t)( 1- t \nu)^{-2} *\eta \|
\\&=
\| (1-t)*( 1- t \nu)^{-1}*\left[( 1- t \nu)^{-1}  -  ( 1- t \nu)^{-1}\right] *\eta \|\leq \varepsilon \| \eta \|,
\end{align*}
where $\delta$ is chosen from $\varepsilon$  by the continuity~\footnote{$( 1- t \lambda)^{-1}=1+t\lambda+ t^2 \lambda ^2+ ...$ is continuous as absolutely convergent series} of $F(\lambda)= ( 1- t \lambda)^{-1}$ in $\nu$.
Following the definition of \cite{Blum}, we get that $d_{\nu}Q_{t } = (1-t)( 1- t \nu)^{-2} \in BM(G).$
\end{proof}

\begin{theorem} \label{infingenerator} The infinitesimal generator of $Q_t$ is $\chi: \mathcal{P}(G) \to BM(G)$ given by
$$ \chi(\mu)= \mu^{2} - \mu.$$
In particular, $\mu$ is a fixed point of $Q_{t }$ only if  $\mu$ is a singularity of the field, i.e., $\chi(\mu)= 0$.
\end{theorem}
\begin{proof}
Let $Q_{t }(\mu)=(1-t) \mu *( 1- t \mu)^{-1} \in\mathcal{P}(G)$ be the flow induced by the rational map. Then time $\triangle$-derivative is given by $\displaystyle\frac{d}{dt_{\triangle}}Q_{t }(\mu)$:
\begin{align*}
\frac{d}{dt_{\triangle}}Q_{t }(\mu)&=\lim_{h\to 0} \frac{1}{h}\left(Q_{t \triangle h}(\mu) - Q_{t }(\mu) \right) =
\lim_{h\to 0} \frac{1}{h}\left(Q_{h}(Q_{t}(\mu)) - Q_{t }(\mu) \right)
\\& =
\lim_{h\to 0} \frac{1}{h}\left(Q_{h}(\nu) - \nu \right)
\end{align*}
where $\nu= Q_{t}(\mu)$.

So,
\begin{align*}
\frac{d}{dt_{\triangle}}Q_{t }(\mu)&=\lim_{h\to 0} \frac{1}{h}\left((1-h) \nu *( 1- h \nu)^{-1} - \nu \right)
\\&= \nu * \lim_{h\to 0} \frac{1}{h}\left((1-h) ( 1- h \nu)^{-1} - 1 \right)
\\&= \nu * \lim_{h\to 0} \frac{1}{h}\left((1-h)  -  ( 1- h \nu) \right) *( 1- h \nu)^{-1}
\\&= - \nu *  ( 1-  \nu)= \nu^2 - \nu= Q_{t }(\mu)^{2} - Q_{t }(\mu),
\end{align*}
what means
$$\frac{d}{dt_{\triangle}}Q_{t }(\mu) + Q_{t }(\mu)= Q_{t }(\mu)^{2}.$$

Let $\chi: \mathcal{P}(G) \to BM(G)$ be the field
$$ \chi(\mu)= \mu^{2} - \mu$$
then  $$\frac{d}{dt_{\triangle}}Q_{t }(\mu)= \chi(Q_{t }(\mu)).$$
We know that $\mu$ is a fixed point of $Q_{t }$ if  $\mu^{2} = \mu$, that is, $ \chi(\mu)= 0$. Thus, the fixed points are the singularities of the field $ \chi$.
\end{proof}

\begin{remark}
We observe that if $\sigma(t)$ is a differentiable path (as a path in a Banach manifold) then,
\begin{align*}
\frac{d \sigma(t)}{dt_{\triangle}}&= \lim_{h\to 0} \frac{1}{h}\left(\sigma(t \triangle h) - \sigma(t) \right)
\\&= \lim_{h\to 0} \frac{1}{h}\left(\frac{\sigma(t \triangle h) - \sigma(t) }{t \triangle h - t} \right) \left(\frac{t \triangle h - t}{h}\right)
=(1-t) \frac{d \sigma(t)}{dt}.
\end{align*}
Thus the derivatives $\frac{d \sigma(t)}{dt_{\triangle}}$ and$\frac{d \sigma(t)}{dt}$ are conformal, by a factor $1-t$.\\
\end{remark}

The next lemma gives a complete characterization of the hyperbolic behavior in the fixed points and it is the key to understand the stability.

\begin{lemma} \label{diff on a fix point} If $\eta$ is a fixed point then
$$d_{\eta}Q_{t } (\nu) =\frac{1}{(1-t)}\left[(1- ( 2t- t^2)) \delta_e + ( 2t- t^2) \eta \right] * \nu.$$
More than that, if $\eta*\mu=\eta$  and $\nu= \mu-\eta$ then
$d_{\eta}Q_{t } (\nu) = (1-t)  \nu.$
\end{lemma}
\begin{proof} We remind the reader that $\eta $ is a fixed point if and only if $\eta^2=\eta$ and, $1-t\eta$ has an inverse which is given by the Liouville series
$$(1 - t \eta)^{-1}= 1+  (t \eta +  t^2  \eta^2 +  t^3 \eta^3 +...)=1+  (t    +  t^2    +  t^3   +...) \eta= 1 + \frac{t}{1-t}\eta.$$
From Lemma~\ref{AB} we get
$$d_{\eta}Q_{t } (\nu) =\left[(1-t) + \frac{( 2t- t^2)}{(1-t)} \eta \right] * \nu,$$
dividing by $1-t$ we get the expression above.
In particular, if $\nu= \mu-\eta$, and $\eta*\mu=\eta$ we get
$$d_{\eta}Q_{t } (\nu) =(1-t)\nu + \left[  \frac{( 2t- t^2)}{(1-t)} \right] \eta * (\mu-\eta)= (1-t)\nu,$$
because $\eta * (\mu-\eta)= \eta *\mu-\eta^2 = \eta *\mu-\eta =0$.
\end{proof}

From the formula of $d_{\eta}Q_{t }$  in the Lemma~\ref{diff on a fix point} there is a hyperbolic behavior, contracting in the ``direction" of $\text{Ker}(\eta)$ and expanding in the ``direction" pointing outside of $\text{Ker}(\eta)$ because $(1-t) \to 0$ and $\frac{( 2t- t^2)}{(1-t)} \to +\infty$ when $t \to 1$. So is natural to consider the stability of  these points.

\begin{center}
\includegraphics[scale=0.4,angle=0]{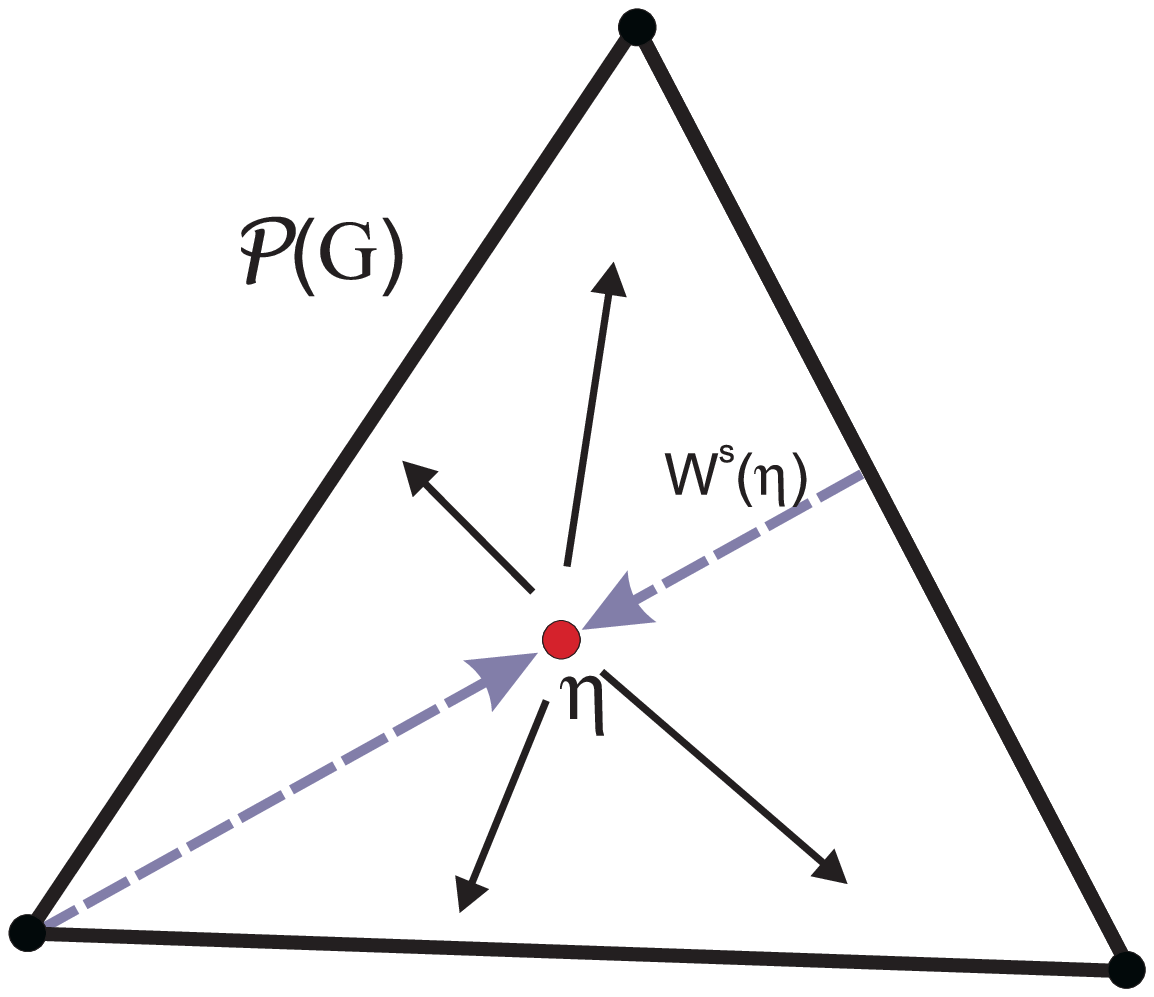}\\
\footnotesize{Stable sets}
\end{center}

\begin{definition}\label{stability definition}
Let $f_{t}: M \to M$ a continuous semigroup of continuous injective maps of a metric space $(M,d)$. The stable set of a fixed point $p$ is the set
$$W^{s}(f,p)=\{ x \in M \; | \; f_{t}(x) \to p, \; t \to +\infty\},$$
moreover, if $U \subset M$ is an open set containing $p$, the local stable set is
$$W^{s}_{loc}(f,p, U)=\{ x \in M \; | \; f_{t}(x) \in U, \; \forall t\geq 0\},$$
in particular $L_{\omega}(x) \subseteq W^{s}(f,p),  \; \forall x \in W^{s}(f,p)$.
\end{definition}

Here, $M=\mathcal{P}(G)$, $d(\mu,\nu)=\|\mu -\nu\|$,  $f_{t}(x)$ is replaced by $Q_{t } (\mu)$, $\eta$ is a fixed point and $U= U_{\delta}(\eta)=\{\mu \in \mathcal{P}(G) \;  | \;  \|(\mu-\eta)\|< \delta\}$. Adapting the local stable set definition to this setting we have
$$W^{s}(\eta)= \{ \mu \in \mathcal{P}(G) \; |\ \; \| Q_t(\mu) - \eta \| \to 0, \; \text{ for }  t \to 1^- \}.$$
$$W^{s}_{\delta}(\eta)= W^{s}_{loc}(\eta, U_{\delta}(\eta))=\{ \mu \in \mathcal{P}(G) \; |\ \; \| Q_t(\mu) - \eta \| < \delta, \; \forall t \in[0 , 1) \}.$$

The next lemma shows that  $Q_t$ is hyperbolic on $\text{Ker}(\eta)$. As an application we can estimate the convergence rate of $Q_t(\mu) \to \eta$ when $t \to 1^-$.
\begin{lemma}\label{exponetial rate stable man} For each $0<\lambda<1$ there is $\delta>0$ and $0<t_0<1$  such that all  $\mu \in  \text{Ker}(\eta) \cap U_{\delta}(\eta)$, we have $\| Q_t(\mu) - \eta \| \leq \lambda \|(\mu-\eta)\|, \; \forall t \in[t_0 , 1).$ Moreover $Q_t(\mu) \in U_{\varepsilon}(\eta), \forall t \in[t_0 , 1)$, for $\varepsilon=\lambda \delta$.  In particular, $Q_{t^{n\Delta}}(\mu) \to \eta$ with the same  rate as $\delta \lambda^n$.
\end{lemma}
\begin{proof}
Consider  $\mu  \in \text{Ker}(\eta) \cap U_{\delta}(\eta)$, we recall that for each $\sigma >0$ there is  $\delta>0 $ such that for all $\mu$ with  $\|\mu-\eta\| <\delta$ we get $\|Q_t(\mu) - \eta - d_{\eta}Q_{t } (\mu-\eta) \|< \sigma \|\mu-\eta\|.$ Since $\mu \in \text{Ker}(\eta)$ we get
$$\| Q_t(\mu) - \eta \| \leq \| Q_t(\mu) - \eta  + d_{\eta}Q_{t } (\mu-\eta) - d_{\eta}Q_{t } (\mu-\eta)\| \leq$$
$$\leq \sigma  \| \mu - \eta \| + \|d_{\eta}Q_{t } (\mu-\eta)\| = \sigma \| \mu - \eta \| +(1-t) \| \mu - \eta \|$$
$$(\sigma +1 -t) \| \mu - \eta \| = \lambda \| \mu - \eta \|,$$
if $\sigma +1 -t = \lambda <1 $ what happens if $t> t_0= 1 + \sigma - \lambda$, then  $Q_t(\mu) \in U_{\varepsilon}(\eta), \forall t \in[t_0 , 1)$, because $\|Q_t(\mu) - \eta\| < \lambda \delta=\varepsilon$.
\end{proof}

\begin{theorem}\label{hyp stab manif general} (The Stable Manifold Theorem)
Given $\eta$ a fixed point of $Q_t$ and  $0<\lambda<1$, $0<\delta $, $0<t_0<1$  given by Lemma~\ref{exponetial rate stable man} then there is $\varepsilon>0,$ such that $W^{s}_{\varepsilon}(\eta)=W^{s}_{loc}(\eta, U_{\varepsilon}(\eta))$  is an invariant local affine submanifold of $\text{Ker}(\eta) \cap U_{\delta}(\eta)$.
\end{theorem}
\begin{proof}
Suppose,  $\mu \in W^{s}_{\varepsilon}(\eta)$. We claim that $\mu \in \text{Ker}(\eta)$, otherwise if $\eta * (\mu-\eta)\neq 0$, by differentiability we have $d_{\eta}Q_{t } (\mu -\eta)= (1-t)(\mu -\eta) + \left[  \frac{( 2t- t^2)}{(1-t)} \right] \eta * (\mu-\eta)$ then the inequality $\|Q_t(\mu) - \eta\| < \varepsilon$ became
$$\|Q_t(\mu) - \eta - d_{\eta}Q_{t } (\mu -\eta) + d_{\eta}Q_{t } (\mu -\eta)\| < \varepsilon$$
$$\|Q_t(\mu) - \eta - d_{\eta}Q_{t } (\mu -\eta) + (1-t)(\mu -\eta) + \left[  \frac{( 2t- t^2)}{(1-t)} \right] \eta * (\mu-\eta)\| < \varepsilon$$
a contradiction because\\
$(1-t)(\mu -\eta) \to 0$, when $t \to 1^-$;\\
$\|Q_t(\mu) - \eta - d_{\eta}Q_{t } (\mu -\eta)\|< \sigma \|\mu-\eta\|$, if $\varepsilon <\delta$;\\
and $\frac{( 2t- t^2)}{(1-t)} \to +\infty$, when $t \to 1^-$. Choosing $\varepsilon <\delta$ small enough we get $W^{s}_{\varepsilon}(\eta) \subset \text{Ker}(\eta)\cap U_{\delta}(\eta)$.
Only remains to show that $\text{Ker}(\eta)\cap U_{\delta}(\eta)$ is a piece of some local stable manifold.

Consider  $\mu  \in Q_{t_0}(\text{Ker}(\eta) \cap U_{\delta}(\eta))=\text{Ker}(\eta) \cap Q_{t_0}(U_{\delta}(\eta))$  because $\text{Ker}(\eta)$ is $Q_{t_0}$ invariant. We are going to show that $\mu \in W^{s}_{\varepsilon}(\eta)$ for $\varepsilon=\lambda \delta< \delta$. Since $\mu  = Q_{t_0}(\nu)$ for some $\nu \in \text{Ker}(\eta) \cap U_{\delta}(\eta)$ we get $ Q_{t }(\mu) = Q_{t_0 \Delta t}(\nu)$ where  $t_0 \Delta t > t_0$ then from Lemma~\ref{exponetial rate stable man} we get  $\|Q_t(\mu) - \eta\| <\|Q_{t_0 \Delta t}(\nu) - \eta\|< \varepsilon$ thus $\mu \in W^{s}_{\varepsilon}(\eta)$. In order to finish  the proof we just point out that $Q_{t_0}(U_{\delta}(\eta))$ is homeomorphic to $U_{\delta}(\eta)$.
\end{proof}

The next result is an immediate corollary of the  exponential convergence in Lemma~\ref{exponetial rate stable man}.
\begin{corollary}\label{global stable manif} Under the same hypothesis of the Theorem~\ref{hyp stab manif general}
there is $\varepsilon>0,$ such that
$$W^{s}(\eta)=\bigcup_{t \geq 0}Q_{t}^{-1} W^{s}_{\varepsilon}(\eta).$$
\end{corollary}

\begin{remark} \label{main repelling} Is easy to see that $W^{s}(\eta)=\{ \mu \; |\ \;  \delta_{e}*\mu=\delta_{e} \} =\{ \delta_{e} \},$
also taking $\eta=\delta_{e}$  we get
$d_{\delta_{e}}Q_{t } (\nu) =\frac{1}{(1-t)}\left[(1- ( 2t- t^2)) \delta_e + ( 2t- t^2) \delta_{e} \right] * \nu= \frac{1}{(1-t)} \nu.$  So the fixed point $\eta=\delta_{e}$ is always a repelling.
\end{remark}

 \begin{example}
 Let $G=\mathbb{S}^1$ and $Leb$ the Lebesgue measure. Then
 $
 W^{s}(Leb)\supseteq \{\mu \in \mathcal{P}(\mathbb{S}^1) \; | \; \mu  \text{ is positive on open sets}\}
 $, from Theorem~\ref{haar}.\\
 \end{example}


\section{Applications to finite commutative groups}\label{applicSection}

In \cite{BOR} the authors has proved an useful computational result on powers of convolution on finite groups.
\begin{definition}\label{Acyclic} (Acyclic)
Given $\nu\in\mathcal{P}(G)$, we define the set $Z_{+}(\nu)^{m}$ by
$$Z_{+}(\nu)^{m}=\{g_{i_1}...g_{i_m}:g_{i_k}\in\mbox{supp}(\nu)\}.$$
Let $H$ the subgroup of $G$ generated by $\mbox{supp}(\nu)$.
We say that $\nu$ is an acyclic probability measure if there exist $N\in\mathbb{N}$ such that
$Z_{+}(\nu)^{N}=H$. In particular, $Z_{+}(\nu)^{1}=\mbox{supp}(\nu)$.
\end{definition}

\begin{theorem}\label{powerconv} (Baraviera-Oliveira-Rodrigues 2014)
Let $G=\{g_0,...,g_{n-1}\}$ be a finite group. If $\nu\in\mathcal{P}(G)$ is an acyclic probability and
$H$ is the subgroup generated by the support of $\nu$, then
$
\lim_{n\to\infty}\nu^{n}=\sum_{h\in H}\frac{1}{|H|}\delta_h.
$
In particular for a dense and open set  in $\mathcal{P}(G)$, $\mbox{supp}(\nu)=G$ then
$
\lim_{n\to\infty}\nu^{n}=\sum_{g\in G}\frac{1}{|G|}\delta_g.
$
\end{theorem}

\subsection{Limit sets and hyperbolicity computations}
\begin{lemma}
Let $G=\{g_0=e,...,g_{n-1}\}$ be a finite group. The Haar measures for $G$
are given by the multiples of the counting measure, i.e., the measures $\mu$ defined by
$\mu=p\sum_{i=0}^{n-1}\delta_{g_i}$, for some $p>0$. In particular, the unique
Haar probability measure on a finite group of order n is given by
$\mu=\frac{1}{n}\sum_{i=0}^{n-1}p_i\delta_{g_i}$.
\end{lemma}
\begin{proof}
Let $E\subset G$ and $\mu=\sum_{i=0}^{n-1}p_i\delta_{g_i}$ a Haar measure on $G$.
Then we know that
$$
\sum_{i=0}^{n-1}p_i\delta_{g_i}(E)=\mu(E)=\mu(g^{-1}E)=\sum_{i=0}^{n-1}p_i\delta_{gg_i}(E)
$$
 If we take $E=\{g\}$ , then
$\mu(\{g\})=\mu(\{e\})$ for all $g\in G$. It implies $p_i=p_j$ for all
$0\leq i,j\leq n-1$ and then $\mu=p\sum_{i=0}^{n-1}\delta_{g_i}$ for $p=p_0$.
\end{proof}

In the next we will discuss the structure of the $\omega$-limit for a finite group.
\begin{lemma}\label{main attrac point}
If G is finite, there is a point $\nu$  such that, its basin of attraction (points that has $\nu$ as its $\omega$-limit), is an open and dense set in $\mathcal{P}(G) $.  We denote this point as the main attractor  point.
\end{lemma}
\begin{proof}
In  \cite{BOR} is shown that there is an open and dense set $\mathcal{O} \subset \mathcal{P}(G)$ where $\mu^{m} \to \overline{\mu}= \sum_{i} 1/n  \delta_{g_i}$ . Let $\mu \in \mathcal{O}$ any point, and $\nu \in L_{\omega}(\mu)$, we know that $\nu*\mu=\nu$, so
$\nu*\mu^{m}=\nu$
taking the limit $m \to \infty$ we get
$\nu * \overline{\mu}=\nu$
that has the unique solution, $\nu= \sum_{i} 1/n  \delta_{g_i}$.
Thus,
$L_{\omega}(\mu)= \{ \sum_{i} 1/n  \delta_{g_i}\}, \; \; \forall \mu \in \mathcal{O}$.\\

\end{proof}

In the next we will compute, explicitly, for $|G|=2$,
$d(Q_{t}^{n}(\mu),  \sum_{i} 1/n  \delta_{g_i}) \;  $ $\forall \mu \in \mathcal{O},$ that is the finite version of Lemma~\ref{exponetial rate stable man}.

\begin{proposition}\label{weak hyperb} (Example using explicit computation) If $G$ is a finite abelian group and $|G|=2$, then the convergence is exponential in basin of attraction of  the main attractor  point $\nu$.
\end{proposition}
\begin{proof}
Let $\mathcal{O} \subset \mathcal{P}(G)$ be the set in Lemma~\ref{main attrac point} and $\nu= \sum_{i} 1/n  \delta_{g_i}$ be the main attractor point. Let $\mu \in \mathcal{O}$ any point. In that case $|G|=2$ thus
$$\mathcal{P}(\mathbb{Z}_{2})=\{\mu= (1-s)   \delta_{0} + s \delta_{1} , | \; s \in [0,1]\} \cong \{ [1-s,  s] \in \mathbb{R}^{2}| \; s \in [0,1]\}.$$

So we can understand $\mathcal{P}(\mathbb{Z}_2)$ as an affine one dimensional  submanifold of $\mathbb{R}^{2}$ parameterized by $\varphi: [0,1] \to \mathbb{R}^{2}$ given by $\varphi(s)=[1-s,  s]$.

Given any $\mu \in \mathcal{O}= \varphi(0  ,  1]$ the only $\omega$-limit point of $Q_{t}$ is $\nu=1/2 \delta_{0} + 1/2 \delta_{1}$. So, we want to estimate the rate of convergence
$\lim_{n \to \infty}  Q_{t}^{n}(\mu) = 1/2 \delta_{0} + 1/2 \delta_{1},$
and show that this convergence is  exponential.
Using $\varphi$, $\nu=1/2 \delta_{0} + 1/2 \delta_{1}= \varphi(1/2)$ so if $\mu= \varphi(s)=[1-s,  s]$ it is enough to show that the second coordinate goes to 1/2.

Let $Q_{t}(\varphi(s))=[1-g(s)  ,  g(s)]$ be the image of $Q_{t}$,  then
$Q_{t}(\varphi(s))= (1-t) \mu * (1- t \mu)^{-1}$
became
$[1-g(s)  ,  g(s)]= (1-t) [1-s,  s]
(I- t \mu(\mathbb{Z}_{2}^{-1} \times \mathbb{Z}_{2}))^{-1}$.

Using
\[\mu(\mathbb{Z}_{2}^{-1} \times \mathbb{Z}_{2})=
\begin{bmatrix}
  1-s & s \\
  s & 1-s \\
\end{bmatrix}
\]
we get,
\begin{align*}
[1-g(s)  ,  g(s)]&= (1-t) [1-s,  s]
\left(\begin{bmatrix}
  1  & 0 \\
  0 & 1 \\
\end{bmatrix}- t
\begin{bmatrix}
  1-s & s \\
  s & 1-s \\
\end{bmatrix}\right)^{-1}\\
&=
(1-t) [1-s,  s]
\left(\begin{bmatrix}
  1-t(1-s) & -t s \\
  -t s & 1-t(1-s) \\
\end{bmatrix}\right)^{-1}.
\end{align*}
Using the inverse,
\begin{align*}
[1-g(s)  ,  g(s)]\left(\begin{bmatrix}
  1-t(1-s) & -t s \\
  -t s & 1-t(1-s) \\
\end{bmatrix}\right) &= (1-t) [1-s,  s].
\end{align*}
Thus, the second coordinate give us,
$(1-g(s))(-t s) + g(s) ( 1-t(1-s))= (1-t)s$
that is equivalent to
$g(s)= \frac{s}{(1-t) + 2 t s}.$
An straight forward computation shows that
$$g'(s)= \frac{(1-t)}{(1-t) + 2 t s}<1\text{ for }s \in [0,1),$$
and $g'(0)=1$ (so the fixed point it is not uniformly contractive).
It is easy to see that $g(1/2)=1/2$ (and $g(0)=0$) so
$$\|g(s) - 1/2\| = \|g(s) - g(1/2)\| \leq g'(s) \|s - 1/2\|.$$
Iterating, we get,
$\|g^{n}(s) - 1/2\| \leq \sigma^{n} \|s - 1/2\|$
where $\displaystyle \sigma= \sup_{n} g'(g^{n}(s))  < 1$.
From this, we get the result since,
\begin{align*}
\|Q_{t}^{n}(\mu) - (1/2 \delta_{0} + 1/2 \delta_{1})\| &= \|[1-g^{n}(s) -(1- 1/2)  ,  g^{n}(s) -1/2]\|\\
& \leq\sigma^{n} \|s - 1/2\|.
\end{align*}
In particular, in any compact subset of $\mathcal{O}$ containing $1/2 \delta_{0} + 1/2 \delta_{1}$,  the convergence is exponential.
\end{proof}

At this point we can prove, for finite groups, that the rational semigroup has exactly one main fixed point that is attracting, the proof is independent of Lemma~\ref{diff on a fix point} that is the equivalent for general groups. But first we need an lemma to show that the derivative on the algebra is the same derivative obtained as a finite dimensional differentiable manifold $\mathcal{P}(G)$ by the standard methods of differential geometry.
\begin{lemma} \label{differential} The linear map  $d_{\mu}Q_{t }: T_{\mu}\mathcal{P}(G) \to T_{Q_{t }(\mu)}\mathcal{P}(G)$ is given by
$$d_{\mu}Q_{t } (\nu) =(1-t)( 1- t \mu )^{-2} * \nu.$$
\end{lemma}
\begin{proof} We recall that for a finite group of order $n$, $\mathcal{P}(G)$ is a differentiable submanifold of dimension $n-1$ of $\mathbb{R}^{n}$. Consider for each fixed $t$ the map, $Q_{t }(\mu): \mathcal{P}(G) \to \mathcal{P}(G)$, then the derivative $d_{\mu}Q_{t }: T_{\mu}\mathcal{P}(G) \to T_{Q_{t }(\mu)}\mathcal{P}(G)$ can be characterized its action on the tangent space. If $\nu= \frac{d\mu (s)}{ds}|_{s=0} \in T_{\mu}\mathcal{P}(G) $, $\mu (0)=\mu$, then
$d_{\mu}Q_{t } \nu = \frac{d}{ds} Q_{t }(\mu (s))|_{s=0}.$

Since the convolution is linear, we can differentiate $Q_t$ implicitly
\begin{align*}
Q_{t }(\mu(s))=(1-t) \mu(s) *( 1- t \mu(s))^{-1} \Rightarrow
Q_{t }(\mu(s))*( 1- t \mu(s))=(1-t) \mu(s)
\end{align*}
differentiating with respect to $s$ and taking $s=0$ we get
\begin{align*}
d_{\mu}Q_{t } (\nu) &*( 1- t \mu ) + Q_{t }(\mu )*( 0 - t \nu)=(1-t) \nu
\\&\Rightarrow
d_{\mu}Q_{t } (\nu) *( 1- t \mu )=((1-t)+ t Q_{t }(\mu ))* \nu
\\&\Rightarrow
d_{\mu}Q_{t } (\nu) *( 1- t \mu )=(1-t)( 1- t \mu )^{-1} * \nu
\\&\Rightarrow
d_{\mu}Q_{t } (\nu) =(1-t)( 1- t \mu )^{-2} * \nu.
\end{align*}
\end{proof}

\begin{theorem}\label{hyperb} If G is finite, then the convergence is exponential in basin of attraction of  the main attractor  point $\eta$.
\end{theorem}
\begin{proof}
Consider $|G|=n \geq 2$. From Theorem~\ref{powerconv} we know that  the main attractor point is
$$
\lim_{k\to\infty}\mu^{k}=\sum_{g\in G}\frac{1}{|G|}\delta_g = \eta.
$$
One must prove that $d_{\eta}Q_{t}$ is a hyperbolic attractive matrix.
The probability $\eta$ is idempotent and we get from Lemma~\ref{differential} that
$$d_{\eta}Q_{t } (\nu) =\left[(1-t) + \frac{( 2t- t^2)}{(1-t)} \eta \right] * \nu.$$
Since $\nu=a_0 \delta_{e} + a_1 \delta_{g_{1}} + ... + a_{n-1}\delta_{g_{n-1}}$ is a tangent measure we have  $a_0 + ... + a_{n-1} =0$. If we write $\left[(1-t) + \frac{( 2t- t^2)}{(1-t)} \eta \right]= b_0 \delta_{e} + b_1 \delta_{g_{1}} + ... + b_{1}\delta_{g_{n-1}}$ where  $b_0 = (1-t) +  \frac{( 2t- t^2)}{n(1-t)}$ and $b_1 = \frac{( 2t- t^2)}{n(1-t)}$.\\
Thus
\begin{align*}
&\left[(1-t) + \frac{( 2t- t^2)}{(1-t)} \eta \right] * \nu
\\&= \left[ b_0 \delta_{e} + b_1 \delta_{g_{1}} + ... + b_{1}\delta_{g_{n-1}} \right] * \left[ a_0 \delta_{e} + a_1 \delta_{g_{1}} + ... + a_{n-1}\delta_{g_{n-1}} \right]
\end{align*}
The coefficient of $\delta_{e}$ is
$$ b_0 a_0 + b_1 ( a_1 + ... + a_{n-1}) = b_0 a_0 + b_1 ( -a_0 )= (b_0 - b_1)* a_0= (1-t))* a_0.$$
Analogously, for $i=1.. n-1$, the coefficient of $\delta_{g_{i}}$ is
\begin{align*}
 b_0 a_i + b_1 ( a_1 + ...  a_{i-1}+  a_{i+1}+ ...  + a_{n-1})&= b_0 a_0 + b_1 ( -a_i )
 \\&= (b_0 - b_1)* a_i= (1-t))* a_i.
 \end{align*}
Thus, $d_{\eta}Q_{t } (\nu) =(1-t) \;\nu,$ an hyperbolic matrix with all eigenvalues equal to $1-t <1$.
\end{proof}

\subsection{Basic Sets} \label{basicsets}
If G is finite, there exist some decomposition  of subgroups $H_i$ of $G$
such that $Q_{t} \mathcal{P}(H_i) \subseteq \mathcal{P}(H_i)$  is an invariant set thus $Q_{t}|_{\mathcal{P}(H_i)}$ has its own main attractor  point $\nu_{i}$  and the convergence is exponential in its basin of attraction (on the induced topology). This decomposition is called \textbf{basic sets}. From this, we can give a complete description of the dynamics of the rational semigroup for the finite cases. In order to illustrate the reach of our technique we will compute the previous elements for the only two groups of order 4, $\mathbb{Z}_4$ and the Klein group. In both cases $P(G)$ can be geometrically represented as a tetrahedron in $\mathbb{R}^{3}$.

\textbf{The case G=$\mathbb{Z}_4$} \\
For $G= \mathbb{Z}_4$, we have\\
$H_0 = \{\overline{0}\}$ and the main attractor is $\nu_{0}=\delta_{\overline{0}}$\\
$H_1=\text{spam}(\overline{2})\sim \mathbb{Z}_2 $ and the main attractor is  $\nu_{1}=1/2 \delta_{\overline{0}} + 1/2 \delta_{\overline{2}}$\\
$H_2=\mathbb{Z}_4$ and the main attractor is $\nu_{2}=1/4 \delta_{\overline{0}} + \cdots + 1/4 \delta_{\overline{3}}$.  A graphical representation is in the figure below.
\begin{center}
\includegraphics[scale=0.5,angle=0]{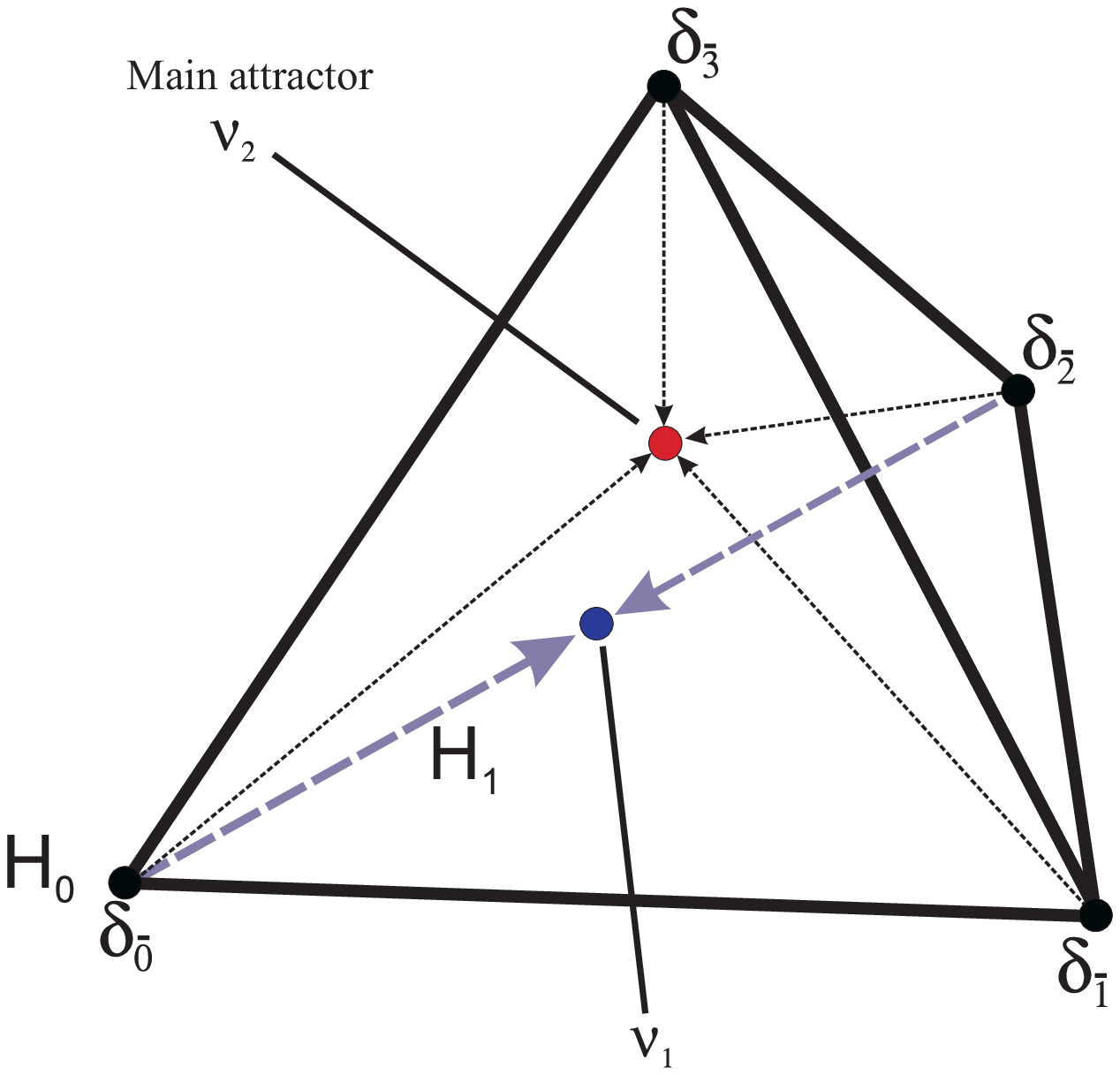}\\
\footnotesize{Spectral decomposition of a cyclic group of order 4}
\end{center}
From Theorem~\ref{hyperb} we cannot have any periodic points except the fixed ones because the flow is contractive.

\textbf{The case G=Klein} \\
If $G$ is not cyclic group, but has order 4, we have a different structure. The computations below are applications of the techniques developed in \cite{BOR}.
For $G= Klein=\{e,a, b, ab\}$, where $a^2=e, b^2=e$ we have cyclical behavior on the complement of the basin of attraction of the main attractor. From Lemma~\ref{main attrac point} we know that almost every point has the main attractor as its $\omega$-limit
$$L_{\omega}(\mu)= \{  \dfrac{1}{4}  \delta_{e} + \dfrac{1}{4}  \delta_{a} +\dfrac{1}{4}  \delta_{b} +\dfrac{1}{4}  \delta_{ab} \}, \; \; \forall \mu \in \mathcal{O},$$
where $\mathcal{O} \supseteq \{x \delta_{e} + y \delta_{a} +z  \delta_{b} +w  \delta_{ab} \; | \; x,y,z >0\}$ is the set of acyclic probabilities, that is geometrically the interior of the tetrahedron  $\mathcal{P}(G)$. Thus, we just need to understand what happens in its complement, that is, faces and edges. Let us to define some basic sets, (see figure below)\\
$\mathcal{P}(G_{a})$ the set of probabilities in $G_{a}=\{e,a\} \sim \mathbb{Z}_{2}$;\\
$\mathcal{P}(G_{b})$ the set of probabilities in $G_{b}=\{e,b\} \sim \mathbb{Z}_{2}$;\\
$\mathcal{P}(G_{ab})$ the set of probabilities in $G_{ab}=\{e,ab\} \sim \mathbb{Z}_{2}$;\\
$A= \{  x  \delta_{a} + y  \delta_{ab} \; | \; x,y >0 \}$,
$B= \{  x  \delta_{b} + y  \delta_{ab} \; | \; x,y >0 \}$,
$C= \{  x  \delta_{a} + y  \delta_{b} \; | \; x,y >0 \}$;\\
$F_{e,a,b}=\{x \delta_{e} + y \delta_{a} +z  \delta_{b} \; | \; x,y,z >0\}$;\\
$F_{e,a,ab}=\{x \delta_{e} + y \delta_{a} +z  \delta_{ab} \; | \; x,y,z >0\}$;\\
$F_{e,b,bb}=\{x \delta_{e} + y \delta_{b} +z  \delta_{ab} \; | \; x,y,z >0\}$;\\
$F_{a,b, ab}=\{x \delta_{a} + y \delta_{b} +z  \delta_{ab} \; | \; x,y,z >0\}$.
\begin{center}
\includegraphics[scale=0.5,angle=0]{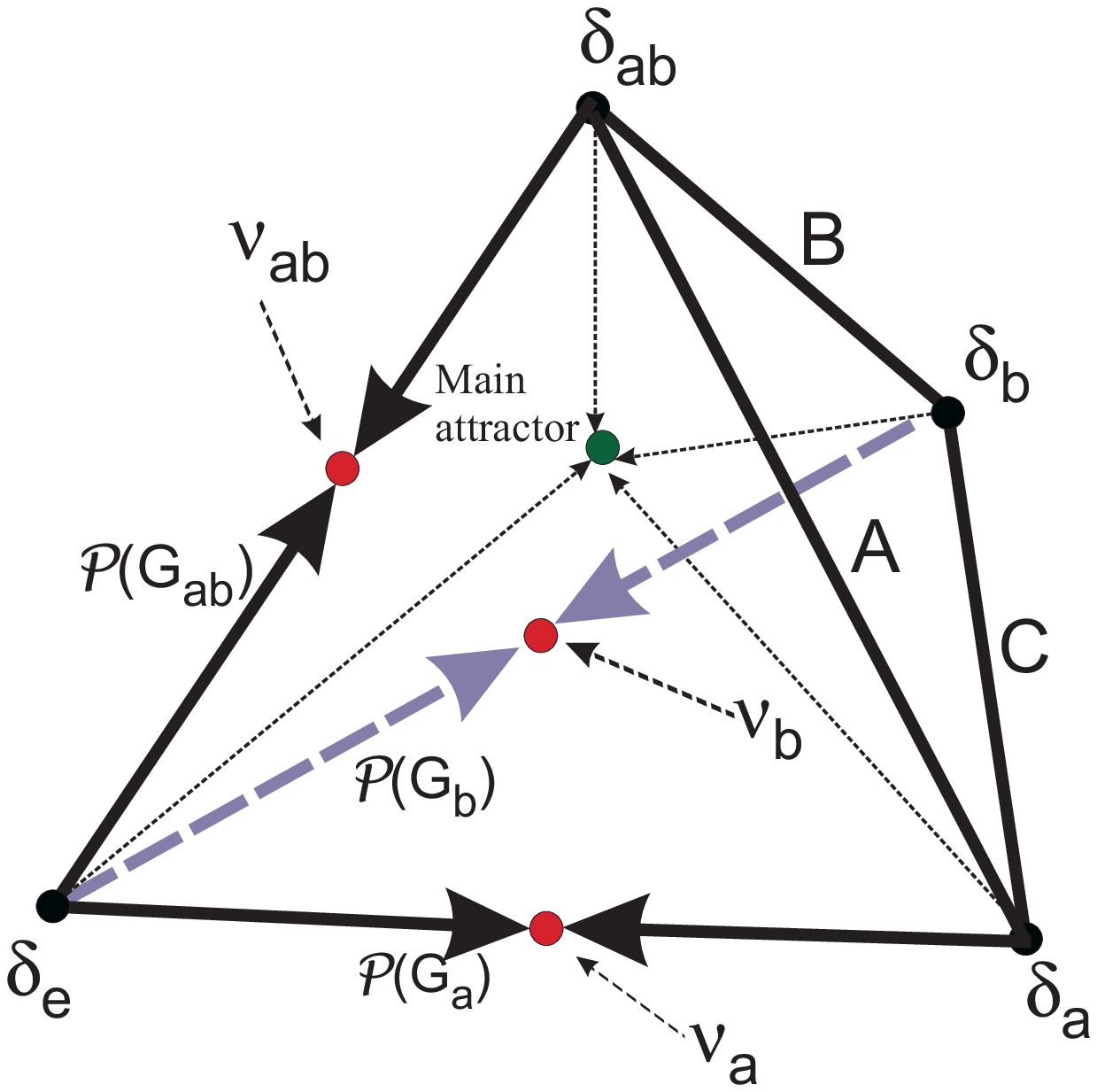}\\
\footnotesize{Spectral Decomposition of $\mathcal{P}(G)$ for a Klein Group of order 4}\\
\end{center}

First, we observe that $\mathcal{P}(G_{a})$, $\mathcal{P}(G_{b})$ and $\mathcal{P}(G_{ab})$ are invariant sets so the restriction of $Q_t$ to them  has the behavior of $\mathbb{Z}_{2}$ what means that the only $\omega$-limit point of $Q_{t}$ is $\nu_{g}=1/2 \delta_{e} + 1/2 \delta_{g}$, given by Proposition~\ref{weak hyperb}, where $g \in \{a, b, ab\}$. In the interior of the  tetrahedron faces $F_{e,a,b}$, ..., $F_{a,b, ab}$ all the probabilities are acyclic, then its unique limit set is the main attractor too.
The last sets to analyze are the interior of the edges $A$, $B$ and $C$. Since $a, b $ and $ab$ play the same role in the Klein group is enough to analyze $A$. \\
\textbf{Claim:} Given $\mu \in A$ its $\omega$-limit is the main attractor.\\
To see that, we observe that the powers of  $\mu$ accumulates in
$\{1/2 \delta_{e} + 1/2 \delta_{b}, \, 1/2 \delta_{a} + 1/2 \delta_{ab}\}$.
Indeed, if $\nu \in L_{\omega}(\mu)$ then $\nu*\mu=\nu$, in particular  $\nu*\mu^n=\nu$ for all $n$. If we compute the powers of $\mu$ (that is not acyclic) a simple computation shows that we have two subsequences:\\
$$
\mu^n \in
\left\{
  \begin{array}{ll}
     \mathcal{P}(G_{b}), & n=2k \\
      A, & n=2k+1
  \end{array}
\right.
$$
Thus, $\mu^{2k} \to 1/2 \delta_{e} + 1/2 \delta_{b}$, and by continuity
$\mu^{2k+1} = \mu^{2k}*\mu \to (1/2 \delta_{e} + 1/2 \delta_{b})*\mu = 1/2 \delta_{a} + 1/2 \delta_{ab}$, then we have a system
$$
\left\{
  \begin{array}{ll}
     \nu*(1/2 \delta_{e} + 1/2 \delta_{b})=\nu \\
      \nu*(1/2 \delta_{a} + 1/2 \delta_{ab})=\nu
  \end{array}
\right.
$$
Writing
\begin{align*}
\nu&= 1/2 \nu + 1/2 \nu = 1/2 (\nu*(1/2 \delta_{e} + 1/2 \delta_{b}))+ 1/2 (\nu*(1/2 \delta_{e} + 1/2 \delta_{b}))
\\&= \nu* (\dfrac{1}{4}  \delta_{e} + \dfrac{1}{4}  \delta_{a} +\dfrac{1}{4}  \delta_{b} +\dfrac{1}{4}  \delta_{ab})= \dfrac{1}{4}  \delta_{e} + \dfrac{1}{4}  \delta_{a} +\dfrac{1}{4}  \delta_{b} +\dfrac{1}{4}  \delta_{ab}.
\end{align*}

\end{document}